%% file: main.tex
\documentclass[11pt,letterpaper]{amsart}
\input{head}

\newcommand{\suma}{\sum_{\alpha=2}^n}
\newcommand{\sumb}{\sum_{\beta=3}^n}

\makeatletter
\@namedef{subjclassname@2020}{\textup{2020} Mathematics Subject Classification}
\makeatother

\begin{document}

\title[5-dim quasi-Einstein manifolds with CSC]{Rigidity of Five-Dimensional quasi-Einstein manifolds with constant scalar curvature}

\author{Zhongxian Cao}
\address{}
\email{}

\begin{abstract}
Let $(M^5,g)$ be a five-dimensional non-trivial simply-connected compact quasi-Einstein manifold with boundary. If $M$ has constant scalar $R$, in \cite{costa2024rigiditycompactquasieinsteinmanifolds} the authors show that $R$ = $((m-5)k+20)/(m-k+4)\lambda$ for some $k\in\{0,2,3,4\}$. Both cases of $k=0$ and $k=4$ are already classified in \cite{HPW-2014} and \cite{costa2024rigiditycompactquasieinsteinmanifolds}. In this paper we will prove that the case $k=3$ is rigid. 
\end{abstract}


\maketitle

\section{introduction}
    A compact $n$-dimensional Riemannian manifold $(M^n, g)$, $\ n \geq 2$, possibly with boundary $\partial M$, is called a \textbf{$m$-quasi-Einstein manifold}, or simply \textbf{quasi-Einstein manifold}, if there exists a smooth potential function $u$ on $M$ satisfying the system
    \begin{equation}\label{Eq1}
        \begin{cases}
            \nabla^2 u=\frac{u}{m}(R i c-\lambda g) & \text { in } M, \\
            u>0 & \text { on } \operatorname{int}(M), \\
            u=0 & \text { on } \partial M,
        \end{cases}
    \end{equation}
    where $\lambda$ and $0<m<\infty$ are constant (cf.\cites{CSW-2011,HPW-2012,HPW-2014}). Here, $\nabla^2 u$ stands for the Hessian of $u$ and Ric is the Ricci tensor of $g$. 
    
    An $m$-quasi-Einstein manifold will be called \textbf{trivial} if $u$ is constant, otherwise it will be \textbf{nontrivial}. Notice that the triviality implies that $M^n$ is an Einstein manifold, and $\partial M =\emptyset$. But conversely an Einstein manifold can be nontrivial, see Table 2.1 in \cite{HPW-2014}. 

    When $m$ is a finite integer, quasi-Einstein manifolds are naturally arise as the base of warped-product Einstein manifolds (see page 267 in \cite{BookBesse} , and also \cite{HPW-2012} for the case with boundary). In particular when $m=1$, assume in addition that $\Delta u = -\lambda u$ to recover the static equation $-(\Delta u)g + \nabla^2 u- u\Ric = 0 $. 

    Choosing $u = e^{-\frac{\varphi}{m}}$, then \eqref{Eq1} takes the form of 
    \begin{equation}\label{Eq2}
        \Ric + \nabla^2 \varphi - \frac1{m}\text{d}\varphi \otimes\text{d}\varphi = \lambda g, \text { in } M. 
    \end{equation}
    When $\partial M=\emptyset$ and $m=\infty$, an quasi-Einstein manifold is precisely a gradient Ricci soliton. The left hand side of the equation \eqref{Eq2} is Bakry-\'Emery Ricci tensor, so quasi-Einstein manifolds are also associated with the study of diffusion operators \cite{B-E-1985} and smooth metric measure space; see related papers \cite{C-2012-smms, C-2012-smms2, M-R-2014, R-2011, W-2011, W-W-2009}. In physics, quasi-Einstein manifolds relate to the geometry of a degenerate Killing horizon and horizon limit; see, e.g. \cite{Phy1, Phy2, Phy3}.

    For explicit examples of quasi-Einstein manifolds, see \cite{BookBesse, exa1, exa2, CSW-2011, C-2012-smms, C-2012-smms2, exa3, exa4, R-2011}. It is natural to seek classification of quasi-Einstein manifolds. He, Petersen and Wylie in \cite{HPW-2014} give a description on \textbf{rigid} quasi-manifolds, that is, Einstein manifolds or manifolds of which universal covering is a product of Einstein manifolds. They proved that a 3-dimensional quasi-Einstein with $m>1$ is rigid if and only if it has constant scalar curvature, and also showed 4-dimensional counterexamples. So the question arises that whether we can classify quasi-Einstein manifolds with constant scalar curvature. Some additional restriction on curvature tensor will help; see, e.g. \cite{HPW-2014, Tflat, Pinchcond}. 
    
    This paper focuses on nontrivial compact quasi-Einstein manifolds with $m>1$, and by \cite{HPW-2012} one has $\lambda>0$. Here are some basic examples:
    \begin{example}\label{example}
        (i). $\mathbb{S}^n_+$, $g = dr^2+\sin^2 rg_{\mathbb{S}^{n-1}}$, $u = \cos r$, where $n\geq1$;

        (ii). $[-\sqrt{\frac{m}{4\lambda}}\pi,\sqrt{\frac{m}{4\lambda}}\pi]\times \mathbb{S}^{n-1}$, $g = dt^2 + \frac{n-2}{\lambda}\sin^2 t g_{\mathbb{S}^{n-1}}$, $u = \cos (\sqrt{\frac{\lambda}{m}}t)$; 

        (iii). $\mathbb{S}^{p+1}_+\times \mathbb{S}^q,\ q\geq2$, equipped with doubly warped product metric
        $$
        g=d r^2+\sin ^2 r g_{\mathbb{S} p}+\frac{q-1}{p+m} g_{\mathbb{S} q}, 
        $$
        where $d r^2+\sin ^2 r g_{\mathbb{S} p}$ is metric on $\mathbb{S}^{p+1}_+$. 
    \end{example}
    
    In \cite{costa2024rigiditycompactquasieinsteinmanifolds}, Costa, Ribeiro Jr. and Zhou  proved the following important theorem: 
    \begin{theorem}[\cite{costa2024rigiditycompactquasieinsteinmanifolds}]\label{Cscval}
    Let $(M^n, g, u, \lambda)$ be a nontrivial compact $m$-quasi-Einstein manifold with boundary, $m>1$ and constant scalar curvature $R$. Then we have:
    \begin{equation}\label{Csc}
        R=\frac{k(m-n)+n(n-1)}{m+n-k-1} \lambda
    \end{equation}
    for some $k \in\{0,2,3, \ldots, n-1\}$.
    \end{theorem}
    \begin{remark}\label{rem0}
    Here $k$ is the dimension of the critical set of $u$. In Example~\ref{example}, each of the three manifolds has constant scalar curvature, and we have $k=0$ in (i), $k=n-1$ in (ii), and $k=q$ in (iii). If $k=0$ and $M$ is simply-connected, then $M$ must be isometric to the standard hemisphere $\mathbb{S}^n_+$ (see \cite[Remark 3]{costa2024rigiditycompactquasieinsteinmanifolds}). 
    \end{remark}

    In \cite{HPW-2014}, He, Petersen and Wylie showed that 3-dimensional simply-connected $m$-quasi-Einstein manifold with boundary and constant scalar curvature must be $\mathbb{S}^3$ or $I\times \mathbb{S}^2$. 
    In \cite{costa2024rigiditycompactquasieinsteinmanifolds} the authors proved the following theorem:
    \begin{theorem}[\cite{costa2024rigiditycompactquasieinsteinmanifolds}]\label{constn-1}
        Let $\left(M^n, g, u, \lambda\right), n \geq 3$, be a nontrivial simply connected compact $m$-quasiEinstein manifold with boundary and $m>1$. Then $M^n$ has constant scalar curvature $R= (n-1) \lambda$ if and only if it is isometric, up to scaling, to the cylinder $\left[0, \frac{\sqrt{m}}{\sqrt{\lambda}} \pi\right] \times N$ with product metric, where $N$ is a compact $\lambda$-Einstein manifold.
    \end{theorem}

    They also proved the following theorem and completed the 4-dimensional classification. This is similar to 4-dimensional shrinking Ricci soliton with constant scalar curvature $R=2\lambda$. 
    \begin{theorem}[\cite{costa2024rigiditycompactquasieinsteinmanifolds}]
        Let $\left(M^4, g, u, \lambda\right)$ be a nontrivial simply connected compact 4-dimensional $m$-quasi-Einstein manifold with boundary and $m>1$. Then $M^4$ has constant scalar curvature $R=2 \frac{(m+2)}{(m+1)} \lambda$ if and only if it is isometric, up to scaling, to the product space $\mathbb{S}_{+}^2 \times \mathbb{S}^2$ with the doubly warped product metric.
    \end{theorem}

    As for $n=5$, $k\in\{0,2,3,4\}$, and both cases of $k=0$ and $k=4$ have been discussed in Remark~\ref{rem0} and Theorem~\ref{constn-1}. In this paper, we prove the following main theorem, which is inspired by \cite{li2025rigidityfivedimensionalshrinkinggradient}. 
    \begin{theorem}\label{THM}
        Let $(M^5, g, u, \lambda)$ be a nontrivial simply-connected compact 5-dimensional $m$-quasi-Einstein manifold with boundary and $m>1$. Then $M^5$ has constant scalar curvature $R= \frac{3m+5}{m+1} \lambda$ if and only if it is isometric, up to scaling, to the product space $\mathbb{S}_{+}^2 \times \mathbb{S}^3$ with the doubly warped product metric.
    \end{theorem}

    In \cite{li2025rigidityfivedimensionalshrinkinggradient}, for Ricci soliton, the authors estimate $\Delta_f (\lambda_1+\lambda_2)$ and Weyl tensor of the level sets to give a series of inequalities. For quasi-Einstein manifolds, the elliptic operator $L_{m+2} = \Delta_{-(m+2)\log u}$ is important, and there are useful identities on the operator $L_{m+2}$, just as $\Delta _f$ does in Ricci solitons. In our paper we seek for estimate on $L_{m+2}(\lambda_1+\lambda_2)$. 
    
    The Weyl tensor of 3-dimensional level set vanishes if manifold $M$ has dimension 4, which appears hard to handle in higher dimensions. Fortunately, in dimension 5 the Gauss-Bonnet-Chern formula can be used to estimate the Weyl tensor of 4-dimensional level set. Actually under our assumption with $k=3$, the Weyl tensor of level sets decays uniformly, which implies that $\partial M$ is locally conformally flat. This also restricts $L_{m+2}(\lambda_1+\lambda_2)$. 

    The rest of this paper is organized as follows. In Section~\ref{sec2}, we collect several useful results that will be used in the proof of Theorem~\ref{THM}. In Section~\ref{sec3}, we give estimates on $L_{m+2}(\mu_1+\mu_2)$. In Section~\ref{sec4}, we calculate $|\nabla \Ric|^2$. In Section~\ref{sec5}, we restrict to $n=5$ and study the uniform decay on $\mu_1+\mu_2$ by use of the Gauss-Bonnet-Chern formula on level sets. In Section~\ref{sec6}, the proof of the main theorem is presented. 

\section{preliminaries}\label{sec2}
    We recall some important properties of $m$-quasi-Einstein manifolds (cf. \cite{CSW-2011,DGR-2022,HPW-2012}, see also \cite{costa2024rigiditycompactquasieinsteinmanifolds}). 
    \begin{lemma}
        Let $(M,g,u, \lambda)$ is a $m$-quasi-Einstein manifold. Then we have:
        \begin{equation}\label{lem1.1}
            \frac{1}{2} u \nabla R=-(m-1) R i c(\nabla u)-(R-(n-1) \lambda) \nabla u;
        \end{equation}
        
        \begin{equation}\label{lem1.2}
            \frac{u^2}{m}(R-n\lambda)+(m-1)|\nabla u|^2+\lambda u^2=\mu=\text{constant};
        \end{equation}
        
        \begin{equation}\label{lem1.3}
            \begin{aligned}
                \frac{1}{2} \Delta R= & -\frac{m+2}{2 u}\langle\nabla u, \nabla R\rangle-\frac{m-1}{m}\left|R i c-\frac{R}{n} g\right|^2 \\
                & -\frac{(n+m-1)}{m n}(R-n \lambda)\left(R-\frac{n(n-1)}{n+m-1} \lambda\right);
            \end{aligned}
        \end{equation}
        
        \begin{equation}\label{lem1.4}
            \begin{aligned}
                u\left(\nabla_i R_{j k}-\nabla_j R_{i k}\right)= & m R_{i j k l} \nabla_l u+\lambda\left(\nabla_i u g_{j k}-\nabla_j u g_{i k}\right)\\
                &-\left(\nabla_i u R_{j k}-\nabla_j u R_{i k}\right) .
            \end{aligned}
        \end{equation}
    \end{lemma}

    \begin{remark}\label{rmk_on_lem}
        If $M$ has constant scalar curvature $R$ and $m>1$, then \eqref{lem1.2} implies that
        \begin{equation}
            \begin{aligned}
                |\nabla u|^2 = & \frac{\mu}{m-1} - \frac{R-(n-m)\lambda}{m(m-1)}u^2\\
                = & B-K u^2,
            \end{aligned}
        \end{equation}
        where $B,K$ are positive constants when $M$ is nontrivial. Together with trace of equation \eqref{Eq1}, that is, $\Delta u = \frac{R-n\lambda}{m}u$, one sees that $u$ is an \textbf{isoparametric} function, so the critical set $N\subset M$ is a minimal submanifold, and every level set ${\Sigma_s}$ (except maximum) is a tubular neighborhood of N (see \cite{Wang1987,Miy2013,Ge2013,Ge2014}for more details). The proof of Theorem~\ref{Cscval} from \cite{costa2024rigiditycompactquasieinsteinmanifolds} also shows that $\dim N = k$; here $k$ appears in \eqref{Csc}.  
    \end{remark}
    
    Now we assume $R$ is constant and $m>1$, then \eqref{lem1.1} implies that
    \begin{equation}\label{Ricnu=}
        \Ric(\nabla u) = \frac{(n-1)\lambda-R}{m-1}\nabla u.
    \end{equation}
    Assume $\rho = \frac{(n-1)\lambda-R}{m-1}$ then $\Ric$, viewed as an $(1,1)$-tensor, has an eigenvalue $\rho$ if $\nabla u \not= 0$. Set 
    \begin{equation}\label{P=}
        P = \Ric - \rho g,    
    \end{equation}
    then from \eqref{Csc},
    \begin{equation}\label{rho}
    \begin{aligned}
        \rho = &\frac{(n-1)\lambda-R}{m-1}\\
        = &\frac{n-k-1}{m+n-k-1}\lambda, 
    \end{aligned}
    \end{equation} 
    and 
    \begin{equation}\label{trP}
        \tr (P) = R - n\rho = k\frac{m\lambda}{m+n-k-1}. 
    \end{equation} 
    By \eqref{lem1.3} and \eqref{Csc}, we have  
    \begin{equation}\label{P^2}
        |P|^2 = k\big(\frac{m\lambda}{m+n-k-1}\big)^2. 
    \end{equation}

    Let eigenvalues of $\Ric$ be $\lambda_1,\ldots,\lambda_n$, then the eigenvalues of $P$ are $\mu_i = \lambda_i - \rho$. Away from critical set of $u$, we know from \eqref{Ricnu=} and \eqref{P=} that $P$ has eigenvalue 0. When $k=n-2$, we may assume $\mu_1=0$, and $\mu_2\leq\ldots\leq\mu_n$, then we have
    \begin{equation}\label{trace}
        \suma \mu_{\alpha} = (n-2)m\rho,
    \end{equation}
    \begin{equation}\label{square}
        \suma \mu_{\alpha}^2 = (n-2)m^2\rho^2,
    \end{equation}
    so
    $$
    \begin{aligned}
        \mu_2 = & (n-2)m\rho - \sumb \mu_\beta\\
        \geq & (n-2)m\rho - \sqrt{(n-2)\sumb \mu_\beta ^2}\\
        = & 0. 
    \end{aligned}
    $$
    Hence we have $0 = \mu_1 \leq \mu_2 \leq \ldots \leq \mu_n$ when $k=n-2$. 

    \begin{definition}
        A $(\lambda, m)$-quasi-Einstein manifold $(M, g, u)$ is called \textbf{rigid} if it is Einstein or its universal cover is a product of Einstein manifolds.
    \end{definition}
    The authors in \cite{HPW-2014} established the description for rigid quasi-Einstein manifold. 
    \begin{proposition}[\!\!\cite{HPW-2014}]\label{Rigid}
        A non-trivial complete rigid $(\lambda, m)$-quasi-Einstein manifold $(M, g, u)$ is one of the examples in Table 2.1 in \cite{HPW-2014}, or its universal cover $\widetilde{M}$ splits as
        $$
        \begin{aligned}
        \widetilde{M} & =\left(M_1, g_1\right) \times\left(M_2, g_2\right) \\
        u(x_1,x_2) & =u(x_2) = u_2(x_2),
        \end{aligned}
        $$
        where $c$ is a constant, $(M_1, g_1, c)$ is a trivial $(\lambda, m)$-quasi-Einstein manifold and $(M_2, g_2, u_2)$ is one of the examples in Table 2.1 in \cite{HPW-2014}.
    \end{proposition}

     The Table 2.1 in \cite{HPW-2014} lists all possible non-trivial quasi-Einstein manifolds, up to scaling of $g$ and $u$. In particular if $M$ is compact, then $(M,g,u) \cong (\mathbb{S}^n_+,g,\cos(r))$, where $g = dr^2 + g_{\mathbb{S}^{n-1}}, \ n\geq1$. 

\section{An estimate of $L_{m+2}(\mu_1+\mu_2)$}\label{sec3}
    In the rest of this paper, we assume $k=k_0=n-2$, $R=\frac{(n-2)m+n}{m+1}\lambda$. Let $L_{m+2} = \Delta + \frac{(m+2)}{u}<\nabla u ,\nabla \cdot>$. By \eqref{rho}, \eqref{trP} and \eqref{P^2} we have
    \begin{equation}\label{lambda}
        \lambda = (m+1)\rho, 
    \end{equation}
    \begin{equation}\label{trace2}
        \tr P = k_0m\rho, 
    \end{equation}
    \begin{equation}\label{square2}
        |P|^2 = k_0m^2\rho^2. 
    \end{equation}

    From \eqref{lem1.4}, we can get the following lemma (see \cite[Lemma 4]{costa2024rigiditycompactquasieinsteinmanifolds}):
    \begin{lemma}[\cite{costa2024rigiditycompactquasieinsteinmanifolds}]\label{lem3.1}
        Let $\left(M^n, g\right)$ be an $n$-dimensional Riemannian manifold satisfying \eqref{Eq1}. Then we have:
        $$
        \begin{aligned}
        u\left(\Delta R_{i k}\right)= & \nabla_i R_{s k} \nabla_s u+m \nabla_k R_{i s} \nabla_s u+\frac{u}{2} \nabla_i \nabla_k R+\frac{1}{2} \nabla_i u \nabla_k R \\
        & +\frac{(m+1)}{m} u R_{i s} R_{s k}+2 u R_{j i k s} R_{j s}-(m+2) \nabla_s R_{i k} \nabla_s u \\
        & -\frac{u}{m}(R-(m+n-2) \lambda) R_{i k}+\frac{\lambda u}{m}(R-(n-1) \lambda) g_{i k}. 
        \end{aligned}
    $$
    \end{lemma}
    
    Let $L_{m+2} = \Delta + \frac{(m+2)}{u}<\nabla u ,\nabla \cdot>$. It follows from Lemma~\ref{lem3.1} that:
    \begin{corollary}\label{LP}
    Let $n\geq4,\ (M^n,g,u, \lambda)$ be a nontrivial compact n-dimensional $m$-quasi-Einstein manifold with boundary and constant scalar curvature $R = \frac{(n-2)m+n}{m+1}\lambda$, $m>1$, then
        \begin{equation}\label{eqLP}
            L_{m+2} P_{ik} = 2(m+1)\rho P_{ik} + 2uR_{jiks}P_{js}. 
        \end{equation}
    \end{corollary}

    \begin{proof}
        Since $\Ric = P + \rho g$, $P(\nabla u) = 0$, and \eqref{Eq1}, we have
        \begin{equation*}
            \begin{aligned}
                \nabla_i R_{sk} \nabla_s u = & \nabla_i P_{sk} \nabla_s u\\
                = & \nabla_i(P_{sk} \nabla_s u) - P_{sk}\nabla_i\nabla_s u\\
                = & -\frac{u}{m}P_{sk}(R_{is}-\lambda g_{is})\\
                = & -\frac{u}{m} P_{sk}(P_{is}-(\lambda-\rho)g_{is})\\
                = & -\frac{u}{m} P_{is}P_{sk} + u\rho P_{ik}, 
            \end{aligned}
        \end{equation*}
        \begin{equation*}
            \begin{aligned}
                &\frac{(m+1)}{m} u R_{i s} R_{s k}  =  \frac{(m+1)u}{m}(P_{is}P_{sk}+2\rho P_{ik} + \rho^2 g_{ik}),
            \end{aligned}
        \end{equation*}
        \begin{equation*}
            \begin{aligned}
                2 u R_{j i k s} R_{j s}  = & 2uR_{jiks}P_{js} - 2u\rho R_{ik}\\
                = & 2uR_{jiks}P_{js} - 2u\rho P_{ik} - 2u\rho^2g_{ik}. 
            \end{aligned}
        \end{equation*}
        
        By $R=\frac{(n-2)m+n}{m+1}\lambda$ and \eqref{lambda}, 
        \begin{equation*}
            \begin{aligned}
                -\frac{u}{m}(R-(m+n-2) \lambda) R_{i k}  =  &u\frac{(m+2)(m-1)}{m}\rho R_{ik}\\
                = &u\frac{(m+2)(m-1)}{m}\rho(P_{ik}+\rho g_{ik}),
            \end{aligned}
        \end{equation*}
        \begin{equation*}
            \begin{aligned}
                &\frac{\lambda u}{m}(R-(n-1) \lambda) g_{i k}  =  -u\frac{(m+1)(m-1)}{m}\rho^2 g_{ik}. 
            \end{aligned}
        \end{equation*}
        Then Lemma~\ref{lem3.1} implies that         
        \begin{equation*}
            \begin{aligned}
                L_{m+2} P_{ik} = &\Delta P_{ik}+ \frac{(m+2)}{u}<\nabla u ,\nabla P_{ik}>\\
                = & -\frac{m+1}{m}P_{is}P_{sk} + (m+1)\rho P_{ik}\\
                & + \frac{m+1}{m}(P_{is}P_{sk}+2\rho P_{ik} + \rho^2 g_{ik})\\
                & + 2R_{jiks}P_{js} - 2\rho P_{ik} - 2\rho^2g_{ik}\\
                & + \frac{(m+2)(m-1)}{m}\rho(P_{ik}+\rho g_{ik})-\frac{(m+1)(m-1)}{m}\rho^2 g_{ik}\\
                = &  2(m+1)\rho P_{ik} + 2uR_{jiks}P_{js}.
            \end{aligned}
        \end{equation*}
    \end{proof}

    At any point $p\in M$, we can choose an orthonormal basis $\{e_1,\ldots,e_n\}$ such that $P_{ij} = \mu_i\delta_{ij}$ at $p$. In a neighborhood $U_p$ of $p$, we can take parallel transport along geodesics from $p$ to get an orthonormal frame. It is easy to check $\big(P(e_1,e_1)+P(e_2,e_2)\big)(x)\geqslant (\mu_1+\mu_2)(x),\ x\in U_p$, with equality holds when $x=p$. From Corollary~\ref{LP} we see that the following holds in the sense of barrier, hence, in distribution (see \cite[Appendix]{wu2025fourdimensionalshrinkersnonnegativericci}): 
    \begin{equation}\label{Lmu1}
        L_{m+2}(\mu_1+\mu_2)\leqslant 2(m+1)\rho(\mu_1+\mu_2) - 2\suma K_{1\alpha}\mu_\alpha - 2\sumb K_{2\beta}\mu_\beta, 
    \end{equation}
    where $K_{ij} = R_{ijij}$ is the sectional curvature of $M$. 

    Now we need to express the sectional curvature. Recall that $\varphi = -m\log u$, and $\nabla\varphi = -mu/\log u$. Now we choose $e_1 = \frac{\nabla\varphi}{|\nabla\varphi|} = -\frac{\nabla u}{|\nabla u|}$ and generate a basis such that $P_{ij} = \mu_i\delta_{ij}$. 

    From \eqref{lem1.4} we get that
    \begin{equation}
        R_{ijkl}\varphi_j\varphi_l = (\nabla_{\nabla \varphi}P_{ik}-P_{ij}(P_{jk}-m\rho g_{jk})) + |\nabla\varphi|^2(\rho-\frac{\mu_\alpha}{m})g_{ik}. 
    \end{equation}
    This formula can be found in the proof of lemma 7 in \cite{costa2024rigiditycompactquasieinsteinmanifolds}(Page.31), and actually holds for all $n\geq4$. 

    Note that $\varphi_1=|\nabla\varphi|$, and $\varphi_\alpha=0$ for $\alpha\geq2$, therefore take $(i,k) = (\alpha,\alpha)$ and $(i,k) = (\alpha,\beta)$ in the formula and get
    
    \begin{lemma}
    Let $n\geq4,\ (M^n,g,u,\lambda)$ be a nontrivial compact n-dimensional $m$-quasi-Einstein manifold with boundary and constant scalar curvature $R = \frac{(n-2)m+n}{m+1}\lambda$, $m>1$, then
        \begin{equation}\label{eqK1a}
            K_{1\alpha} = \frac1{|\nabla\varphi|^2}(\nabla_{\nabla\varphi}P_{\alpha\alpha} - \mu_\alpha(\mu_\alpha-m\rho)) + (\rho - \frac{\mu_\alpha}{m}), 
        \end{equation}
        \begin{equation}\label{eqR1a1b}
            R_{1\alpha1\beta} = \frac1{|\nabla\varphi|^2}\nabla_{\nabla\varphi}P_{\alpha\beta}, \ \ \alpha\not=\beta.  
        \end{equation}
    \end{lemma}
    
    \begin{corollary}\label{K1amua}
    Let $n\geq4,\ (M^n,g,u,\lambda)$ be a nontrivial compact n-dimensional $m$-quasi-Einstein manifold with boundary and constant scalar curvature $R = \frac{(n-2)m+n}{m+1}\lambda$, $m>1$, then
        \begin{equation}\label{eqK1amua}
        \begin{aligned}
            -\suma K_{1\alpha}\mu_\alpha = \frac1{|\nabla\varphi|^2}\suma\mu_\alpha(\mu_\alpha-m\rho)^2. 
        \end{aligned}
        \end{equation}
    \end{corollary}
    \begin{proof}
        From \eqref{trace2} and \eqref{square2}, it is easy to get that $$\suma \mu_\alpha(\mu_\alpha-m\rho) =0,$$and $$\suma \mu_\alpha^2(\mu_\alpha-m\rho) = \suma \mu_\alpha(\mu_\alpha-m\rho)^2.$$

        Since $|P|^2$ is constant, $\nabla|P|^2 = \suma \mu_\alpha\nabla(P_{\alpha\alpha}) = 0$. So we get         
        \begin{equation*}
        \begin{aligned}
            &-\suma K_{1\alpha}\mu_\alpha\\
            = & -\suma \mu_\alpha\Big( \frac1{|\nabla\varphi|^2}(\nabla_{\nabla\varphi}P_{\alpha\alpha} - \mu_\alpha(\mu_\alpha-m\rho)) + (\rho - \frac{\mu_\alpha}{m}) \Big)\\ 
            = & -\suma  \frac{\mu_\alpha\nabla_{\nabla\varphi}P_{\alpha\alpha}}{|\nabla\varphi|^2} + \frac1{|\nabla\varphi|^2}\suma\mu_\alpha^2(\mu_\alpha-m\rho)-\frac1{m}\suma \mu_\alpha(m\rho-\mu_\alpha)\\
            = & \frac1{|\nabla\varphi|^2}\suma\mu_\alpha^2(\mu_\alpha-m\rho)\\
            = & \frac1{|\nabla\varphi|^2}\suma\mu_\alpha(\mu_\alpha-m\rho)^2. 
        \end{aligned}
        \end{equation*}
    \end{proof}

    On the level set $\Sigma_s$ of $\varphi$, i.e. $\Sigma_s = \varphi^{-1}(s)$, for any regular value $s$ of $\varphi$, using \eqref{Eq2}, \eqref{lambda} and \eqref{trace2} we have:
    \begin{equation}\label{hab}
    \begin{aligned}
        h_{\alpha\beta} &= \frac{-\varphi_{\alpha\beta}}{|\nabla\varphi|}\\
        &= \frac1{|\nabla\varphi|}(R_{\alpha\beta}-\lambda g_{\alpha\beta}) \\
        &= \frac1{|\nabla\varphi|}(\mu_\alpha - m\rho)g_{\alpha\beta},\ 2\leq\alpha,\beta\leq n;
    \end{aligned}
    \end{equation}
    \begin{equation}
    \begin{aligned}
        H = &\suma \frac1{|\nabla\varphi|}(\mu_\alpha - m\rho)\\
        = &\frac1{|\nabla\varphi|}(k_0m\rho - (n-1)m\rho) \\
        = &\frac{-m\rho}{|\nabla\varphi|},
    \end{aligned}
    \end{equation}
    \begin{equation}
        |A|^2 = \sum_{\alpha,\beta}(h_{\alpha\beta})^2 = \frac{m^2\rho^2}{|\nabla\varphi|^2} = H^2.
    \end{equation}

    From Gauss equation and \eqref{hab}, 
    \begin{equation}\label{eqrm}
    \begin{aligned}
        R^{{\Sigma_s}}_{\alpha\beta\alpha\beta} =& R_{\alpha\beta\alpha\beta} + h_{\alpha\alpha}h_{\beta\beta} - (h_{\alpha\beta})^2\\
        = &R_{\alpha\beta\alpha\beta} + \frac1{|\nabla\varphi|^2}(\mu_\alpha-m\rho)(\mu_\beta-m\rho).
    \end{aligned}
    \end{equation}
    Using \eqref{eqK1a}, we have
    \begin{equation}\label{eqric}
    \begin{aligned}
        R^{{\Sigma_s}}_{\alpha\alpha} = &R_{\alpha\alpha} - K_{1\alpha} + Hh_{\alpha\alpha} - (h_{\alpha\alpha})^2 \\
        = &(1+\frac1{m})\mu_\alpha - \frac{\nabla_{\nabla\varphi}P_{\alpha\alpha}}{|\nabla\varphi|^2},
    \end{aligned}
    \end{equation}
    \begin{equation}\label{eqsc}
    \begin{aligned}
        R^{\Sigma_s} = &R - 2R_{11} + H^2 - |A|^2 \\
        = &(n-2)(m+1)\rho. 
    \end{aligned}
    \end{equation}

    Noting
    \begin{equation}\label{eqweyl}
        R^{{\Sigma_s}}_{\alpha\beta\alpha\beta} = W_{\alpha\beta}^{\Sigma_s} + \frac1{n-3}(R^{{\Sigma_s}}_{\alpha\alpha} + R^{{\Sigma_s}}_{\beta\beta}) - \frac1{(n-2)(n-3)}R^{\Sigma_s},
    \end{equation}
    where $W_{\alpha\beta}^{\Sigma_s} = W_{\alpha\beta\alpha\beta}^{\Sigma_s}$ is the Weyl tensor of level set, from \eqref{eqric}, \eqref{eqrm} and  \eqref{eqweyl} we have:
    \begin{lemma}
    Let $n\geq4,\ (M^n,g,u,\lambda)$ be a nontrivial compact n-dimensional $m$-quasi-Einstein manifold with boundary and constant scalar curvature $R = \frac{(n-2)m+n}{m+1}\lambda$, $m>1$, then for $2\leq \alpha,\beta\leq n$, $(\alpha\not= \beta)$
        \begin{equation}\label{Kab}
        \begin{aligned}
            K_{\alpha\beta} = &W_{\alpha\beta}^{\Sigma_s} + \frac1{n-3}(1+\frac1{m})(\mu_\alpha + \mu_\beta - m\rho) \\
            & - \frac1{|\nabla\varphi|^2}\Big(\frac1{n-3}\nabla_{\nabla \varphi}(P_{\alpha\alpha}+P_{\beta\beta})+(\mu_\alpha-m\rho)(\mu_\beta-m\rho)\Big). 
        \end{aligned}
        \end{equation}
    \end{lemma}

    \begin{corollary}\label{K2bmub}
    Let $n\geq4,\ (M^n,g,u,\lambda)$ be a nontrivial compact n-dimensional $m$-quasi-Einstein manifold with boundary and constant scalar curvature $R = \frac{(n-2)m+n}{m+1}\lambda$, $m>1$, then
        \begin{equation}\label{eqK2bmub}
        \begin{aligned}
             -2 \sum_{\beta=3}^n K_{2 \beta} \mu_\beta   = & -2 \sum_{\beta=3}^n W_{2 \beta}^{\Sigma_s} \mu_\beta\\
            & +\frac{2}{n-3} \frac{1}{|\nabla \varphi|^2}\left(k_0 m \rho \nabla \varphi \cdot \nabla P_{22}-\nabla \varphi \cdot \nabla P_{22}^2\right) \\
            & -\frac{2}{|\nabla \varphi|^2} \mu_2\left(\mu_2-m \rho\right)^2 \\
            & +\frac{2}{n-3} \frac{m+1}{m} \mu_2\left(2 \mu_2-(n-1) m \rho\right). 
        \end{aligned}
        \end{equation}
    \end{corollary}

    \begin{proof}By using \eqref{trace2} and \eqref{square2} we get: 
        \begin{equation}\label{cor361}
            \begin{aligned}
                \sumb (\mu_2+\mu_\beta-m\rho)\mu_{\beta} = &  \mu_2(k_0m\rho-\mu_2) + (k_0m^2\rho^2-\mu_2^2) - m\rho(k_0m\rho-\mu_2)\\
                = & -\mu_2(2\mu_2-(n-1)m\rho),
            \end{aligned}
        \end{equation}
        \begin{equation}\label{cor362}
            \begin{aligned}
                \sumb (\mu_2-m\rho)(\mu_\beta-m\rho)\mu_{\beta} = & (\mu_2-m\rho)\Big(\suma(\mu_\alpha^2-m\rho\mu_\alpha)-\mu_2(\mu_2-m\rho)\Big)\\
                = & -\mu_2(\mu_2-m\rho)^2,
            \end{aligned}
        \end{equation}
        \begin{equation}\label{cor363}
            \begin{aligned}
                \sumb\mu_\beta\nabla_{\nabla \varphi}(P_{22}+P_{\beta\beta}) = & \Big(k_0m\rho\nabla_{\nabla \varphi}P_{22}-\mu_2\nabla_{\nabla \varphi}P_{22}\Big)  \\
                & + \Big( \sum_{i,j=1}^n \mu_i\delta_{ij}\nabla_{\nabla \varphi}P_{ij} - \mu_2\nabla_{\nabla \varphi}P_{22}\Big) \\ 
                = & \big(k_0m\rho\nabla_{\nabla \varphi}P_{22}-\nabla_{\nabla \varphi}P_{22}^2\big) + \frac1{2}\nabla_{\nabla \varphi}|P|^2 \\
                = & k_0m\rho\nabla_{\nabla \varphi}P_{22}-\nabla_{\nabla \varphi}P_{22}^2. 
            \end{aligned}
        \end{equation}
        Putting \eqref{cor361}-\eqref{cor363} into \eqref{Kab}, we complete the proof of Corollary~\ref{K2bmub}. 
    \end{proof}

    From~\ref{eqK1amua} and~\ref{eqK2bmub} and \eqref{Lmu1} we have:
    \begin{proposition}
        Let $n\geq4,\ (M^n,g,u,\lambda)$ be a nontrivial compact n-dimensional $m$-quasi-Einstein manifold with boundary and constant scalar curvature $R = \frac{(n-2)m+n}{m+1}\lambda$, $m>1$, then
        \begin{equation}\label{Lmu2}
        \begin{aligned}
            L_{m+2}(\mu_1+\mu_2)\leqslant & - \frac{4(m+1)}{(n-3)m}\mu_2(m\rho-\mu_2)+ \frac{2}{|\nabla\varphi|^2}\sumb\mu_\beta(\mu_\beta-m\rho)^2\\
            & - 2\sumb W_{2\beta}^{\Sigma_s}\mu_\beta + \frac2{n-3}\frac1{|\nabla\varphi|^2}(k_0m\rho\nabla_{\nabla \varphi}P_{22}-\nabla_{\nabla \varphi}P_{22}^2), 
        \end{aligned}
        \end{equation}
        in $M\backslash N$ in the sense of distribution, where $N$ is the critical set of $\varphi$. 
    \end{proposition}

    We also need the following lemma:
    \begin{lemma}
    Let $n\geq4,\ (M^n,g,u,\lambda)$ be a nontrivial compact n-dimensional $m$-quasi-Einstein manifold with boundary and constant scalar curvature $R = \frac{(n-2)m+n}{m+1}\lambda$, $m>1$, then
        \begin{equation}
            \sumb \mu_\beta(\mu_\beta-m\rho)^2 \leq 2k_0m^2\rho^2\mu_2, 
        \end{equation}
        \begin{equation}\label{ineqb2b2}
            \sumb \mu_\beta^2(\mu_\beta-m\rho)^2 \leq 2k_0m^3\rho^3\mu_2.             
        \end{equation}
    \end{lemma}
    \begin{proof}
        Since $\mu_i\geq 0$, we use \eqref{trace2} and \eqref{square2} to derive the following estimates:
        \begin{equation*}
            \begin{aligned}
                &\sumb \mu_\beta(\mu_\beta-m\rho)^2 \\
                \leq& (\sumb \mu_\beta)(\sumb (\mu_\beta-m\rho)^2)\\
                =&(k_0m\rho-\mu_2)(k_0m^2\rho^2 - 2m\rho(k_0m\rho-\mu_2)+k_0m^2\rho^2-\mu_2^2)\\
                =&(k_0m\rho-\mu_2)(2m\rho-\mu_2)\mu_2\\
                \leq&2k_0m^2\rho^2\mu_2,
            \end{aligned}
        \end{equation*}
        \begin{equation*}
            \begin{aligned}
                &\sumb \mu_\beta^2(\mu_\beta-m\rho)^2 \\
                \leq& (\sumb \mu_\beta^2)(\sumb (\mu_\beta-m\rho)^2)\\
                =&(k_0m^2\rho^2-\mu_2^2)(k_0m^2\rho^2 - 2m\rho(k_0m\rho-\mu_2)+k_0m^2\rho^2-\mu_2^2)\\
                =&(k_0m^2\rho^2-\mu_2^2)(2m\rho-\mu_2)\mu_2\\
                \leq&2k_0m^3\rho^3\mu_2. 
            \end{aligned}
        \end{equation*}
    \end{proof}

\section{An estimate of $|\nabla P|^2$}\label{sec4}
    In this section we will give an estimate of $|\nabla P|^2$. 

    \begin{lemma}
        Let $n\geq4,\ (M^n,g,u,\lambda)$ be a nontrivial compact n-dimensional $m$-quasi-Einstein manifold with boundary and constant scalar curvature $R = \frac{(n-2)m+n}{m+1}\lambda$, $m>1$, then
        \begin{equation}\label{np2}
            |\nabla P|^2 = 2\sum_{\alpha\not=\beta}K_{\alpha\beta}\mu_\alpha\mu_\beta - 2k_0(m+1)m^2\rho^3. 
        \end{equation}
    \end{lemma}

    \begin{proof}
        From \eqref{eqLP}, we have
        \begin{equation*}
            L_{m+2} P_{ij} = 2(m+1)\rho P_{ij} + 2R_{kijs}P_{ks}. 
        \end{equation*}
        Since $|P|^2$ is constant, $0=\frac1{2}L_{m+2}|P|^2 = |\nabla P|^2 + P_{ij}L_{m+2}P_{ij}$, so
        \begin{equation*}
            \begin{aligned}
                |\nabla P|^2 = & \sum_{i,j}-P_{ij}L_{m+2}P_{ij}\\
                =& 2\sum_{i,j} K_{ij}\mu_i\mu_j - 2(m+1)\rho\sum_{i=1}^n\mu_i^2\\
                =& 2\sum_{\alpha\not=\beta}K_{\alpha\beta}\mu_\alpha\mu_\beta - 2k_0(m+1)m^2\rho^3. 
            \end{aligned}
        \end{equation*}
    \end{proof}

    Now we get the following estimate: 
    \begin{lemma}
        Let $n\geq4,\ (M^n,g,u,\lambda)$ be a nontrivial compact n-dimensional $m$-quasi-Einstein manifold with boundary and constant scalar curvature $R = \frac{(n-2)m+n}{m+1}\lambda$, $m>1$, then there exists a constant $C= C(n,m,\rho)>0$, such that 
        \begin{equation}\label{inNP2_1}
        \begin{aligned}
            |\nabla P|^2 \leq& \frac{C}{|\nabla\varphi|^2}|\nabla P|^2 + C\mu_2(1+\frac1{|\nabla\varphi|^2}) \\
            &+4\sum_{2\leq\alpha<\beta\leq n}W_{\alpha\beta}^{\Sigma_s}\mu_\alpha\mu_\beta + \frac{4m^2\rho^2}{n-3}(K_{12}-\rho). 
        \end{aligned}
        \end{equation}
    \end{lemma}
    \begin{proof}
        From \eqref{Kab}, we have
        \begin{equation*}
            \begin{aligned}
                &2\sum_{\alpha\not=\beta}K_{\alpha\beta}\mu_\alpha\mu_\beta\\
                = & 4\sum_{2\leq\alpha<\beta\leq n}W_{\alpha\beta}^{\Sigma_s}\mu_\alpha\mu_\beta \\
                & {+ 2\sum_{\alpha\not=\beta}\frac1{n-3}(1+\frac1{m})(\mu_\alpha+\mu_\beta-m\rho)\mu_\alpha\mu_\beta} \\
                & {- 2\sum_{\alpha\not=\beta}\frac1{|\nabla\varphi|^2}\Big(\frac1{n-3}\nabla_{\nabla \varphi}(P_{\alpha\alpha}+P_{\beta\beta})\Big)\mu_\alpha\mu_\beta}\\
                & {- 2\sum_{\alpha\not=\beta}\frac1{|\nabla\varphi|^2}(\mu_\alpha-m\rho)(\mu_\beta-m\rho)\mu_\alpha\mu_\beta}. 
            \end{aligned}
        \end{equation*}
        Notice that $(\mu_\alpha+\mu_\beta)\mu_\alpha\mu_\beta$ is symmetric in $\alpha$ and $\beta$, 
        \begin{equation*}
            \begin{aligned}
                {I} \coloneqq &{{\sum_{\alpha\not=\beta}(\mu_\alpha+\mu_\beta-m\rho)\mu_\alpha\mu_\beta}} \\
                =& \suma \sum_{\beta\not=\alpha} \big(2\mu_\alpha^2\mu_\beta - m\rho\mu_\alpha\mu_\beta\big). 
            \end{aligned}
        \end{equation*}
        Since $m\rho\suma \mu_\alpha = k_0m^2\rho^2 = \suma\mu_\alpha^2$, and $\mu_i\geq0$, we have
        \begin{equation*}
            \begin{aligned}
                {I} =& \suma \sum_{\beta\not=\alpha} \big(2\mu_\alpha^2\mu_\beta - m\rho\mu_\alpha\mu_\beta\big)\\
                =& \suma \big(2\mu_\alpha^2(k_0m\rho-\mu_\alpha)-m\rho\mu_\alpha(k_0m\rho-\mu_\alpha) \big)\\
                =& -2\suma\mu_\alpha^3+k_0(k_0+1)m^3\rho^3\\
                =& -2\suma\mu_\alpha(\mu_\alpha-m\rho)^2 + k_0(k_0-1)m^3\rho^3\\
                \leq& k_0(k_0-1)m^3\rho^3\\
                =& (n-3)k_0m^3\rho^3. 
            \end{aligned}
        \end{equation*}

        Similarly $(P_{\alpha\alpha}+P_{\beta\beta})\mu_\alpha\mu_\beta$ is symmetric in $\alpha$ and $\beta$, 
        \begin{equation*}
            \begin{aligned}
            {II} \coloneqq &{{-\sum_{\alpha\not=\beta}\frac{\nabla_{\nabla \varphi}(P_{\alpha\alpha}+P_{\beta\beta})}{|\nabla\varphi|^2}\mu_\alpha\mu_\beta}}\\
            =&-2\suma\sum_{\beta\not=\alpha}\frac{\nabla_{\nabla \varphi}P_{\alpha\alpha}}{|\nabla\varphi|^2}\mu_\alpha\mu_\beta. 
            \end{aligned}
        \end{equation*}
        Notice that $\nabla\tr(P)=\suma \nabla P_{\alpha\alpha}=0$, and $\nabla|P|^2=\suma \nabla P_{\alpha\alpha}\mu_\alpha=0$, 
        \begin{equation*}
            \begin{aligned}
            {II} =&-2\suma\sum_{\beta\not=\alpha}\frac{\nabla_{\nabla \varphi}P_{\alpha\alpha}}{|\nabla\varphi|^2}\mu_\alpha\mu_\beta\\
            =&-2\suma\frac{\nabla_{\nabla \varphi}P_{\alpha\alpha}}{|\nabla\varphi|^2}\mu_\alpha(k_0m\rho-\mu_\alpha)\\
            =&\frac{2}{|\nabla\varphi|^2}\suma(-k_0m\rho\nabla_{\nabla \varphi}P_{\alpha\alpha}\mu_\alpha +\nabla_{\nabla \varphi}P_{\alpha\alpha}\mu_\alpha^2)\\
            =&\frac{2}{|\nabla\varphi|^2}\suma(m^2\rho^2\nabla_{\nabla \varphi}P_{\alpha\alpha}-2m\rho\nabla_{\nabla \varphi}P_{\alpha\alpha}\mu_\alpha +\nabla_{\nabla \varphi}P_{\alpha\alpha}\mu_\alpha^2)\\
            =&2\suma\frac{\nabla_{\nabla \varphi}P_{\alpha\alpha}}{|\nabla\varphi|^2}(m\rho-\mu_\alpha)^2. \\
            \end{aligned}
        \end{equation*}

        We have the following calculation: 
        \begin{equation*}
            \begin{aligned}
            &2\suma\frac{\nabla_{\nabla \varphi}P_{\alpha\alpha}}{|\nabla\varphi|^2}(m\rho-\mu_\alpha)^2\\
            =&2\sumb\frac{\nabla_{\nabla \varphi}P_{\beta\beta}}{|\nabla\varphi|^2}(m\rho-\mu_\beta)^2+2\frac{\nabla_{\nabla \varphi}P_{22}}{|\nabla\varphi|^2}\Big(\mu_2(\mu_2-2m\rho)\Big)+2m^2\rho^2\frac{\nabla_{\nabla \varphi}P_{22}}{|\nabla\varphi|^2}\\
            \leq& (n-2)\frac{|\nabla P|^2}{|\nabla\varphi|^2} + (\sumb(m\rho-\mu_\beta)^2)^2 +\frac{|\nabla P|^2}{|\nabla\varphi|^2}+\mu_2^2(\mu_2-2m\rho)^2 + 2m^2\rho^2\frac{\nabla_{\nabla \varphi}P_{22}}{|\nabla\varphi|^2}\\
            = & (n-1)\frac{|\nabla P|^2}{|\nabla\varphi|^2}+ 2\mu_2^2(\mu_2-2m\rho)^2 + 2m^2\rho^2\frac{\nabla_{\nabla \varphi}P_{22}}{|\nabla\varphi|^2}. 
            \end{aligned}
        \end{equation*}
        Taking $\alpha=2$ in \eqref{eqK1a}, we have
        \begin{equation*}
            \frac{\nabla_{\nabla \varphi}P_{22}}{|\nabla\varphi|^2} = K_{12}-\rho+\frac{1}{m}\mu_2+\frac{\mu_2-m\rho}{|\nabla\varphi|^2}\mu_2, 
        \end{equation*}
        so we get
        \begin{equation*}
            \begin{aligned}
            II\leq& (n-1)\frac{|\nabla P|^2}{|\nabla\varphi|^2}+ 2\mu_2^2(\mu_2-2m\rho)^2\\
            &+ \mu_2 2m^2\rho^2(\frac1{m} + \frac{\mu_2-m\rho}{|\nabla\varphi|^2})+ 2m^2\rho^2(K_{12}-\rho)\\
            \leq&(n-1)\frac{|\nabla P|^2}{|\nabla\varphi|^2} + \mu_2 (8m^3\rho^3 + 2m\rho^2)\\
            &+ 2m^2\rho^2(K_{12}-\rho), 
            \end{aligned}
        \end{equation*}
        where the last inequality due to $\mu_2\leq\frac{\tr P-\mu_1}{n-1}=\frac{k_0m\rho}{n-1}<m\rho$. 

        \begin{equation*}
            \begin{aligned}
                {III} = &{- \sum_{\alpha\not=\beta}(\mu_\alpha-m\rho)(\mu_\beta-m\rho)\mu_\alpha\mu_\beta} \\
                =& -\suma\mu_\alpha(\mu_\alpha-m\rho)(k_0m^2\rho^2-\mu_\alpha^2-k_0m^2\rho^2+m\rho\mu_\alpha)\\
                =& \suma\mu_\alpha^2(\mu_\alpha-m\rho)^2\\
                \leq& 2k_0m^3\rho^3\mu_2 + \mu_2^2(\mu_2-m\rho)^2\\
                \leq& (2n-3)m^3\rho^3\mu_2, 
            \end{aligned}
        \end{equation*}
        where we used \eqref{ineqb2b2} and $\mu_2<m\rho$.  

        From the above calculation and \eqref{np2}, we have
        \begin{equation*}
            \begin{aligned}
                |\nabla P|^2 =& 2\sum_{\alpha\not=\beta}K_{\alpha\beta}\mu_\alpha\mu_\beta - 2k_0(m+1)m^2\rho^3\\
                =& 4\sum_{2\leq\alpha<\beta\leq n}W_{\alpha\beta}^{\Sigma_s}\mu_\alpha\mu_\beta- 2k_0(m+1)m^2\rho^3\\
                & {+ \frac{2}{n-3}(1+\frac1{m})I}+ {\frac{2}{n-3}II} + {\frac{2}{|\nabla\varphi|^2}III}\\
                \leq& 4\sum_{2\leq\alpha<\beta\leq n}W_{\alpha\beta}^{\Sigma_s}\mu_\alpha\mu_\beta- 2k_0(m+1)m^2\rho^3 {+ 2k_0(m+1)m^2\rho^3} \\
                & {+2\frac{n-1}{n-3}\frac{|\nabla P|^2}{|\nabla\varphi|^2} + \frac{\mu_2 }{n-3}(16m^3\rho^3 + 4m\rho^2)+ \frac{4m^2\rho^2}{n-3}(K_{12}-\rho)}\\
                & {+\frac{(4n-6)m^3\rho^3\mu_2}{|\nabla\varphi|^2}}\\
                =&\frac{2n-2}{n-3}\frac{|\nabla P|^2}{|\nabla\varphi|^2} + 4\sum_{2\leq\alpha<\beta\leq n}W_{\alpha\beta}^{\Sigma_s}\mu_\alpha\mu_\beta + \frac{4m^2\rho^2}{n-3}(K_{12}-\rho)\\
                &+ \mu_2( \frac{16m^3\rho^3+4m\rho^2}{n-3}+(4n-6)m^3\rho^3\frac{1}{|\nabla\varphi|^2}).
            \end{aligned}
        \end{equation*}
    \end{proof}

    Now we consider the terms with Weyl tensor (see Claim 4 in Proposition 4.2 of \cite{costa2024rigiditycompactquasieinsteinmanifolds}):
    \begin{lemma}\label{ineqW}For any $\epsilon>0$, we have
        \begin{equation}\label{eqW2b}
            \sumb W_{2\beta}^{\Sigma_s}\mu_\beta \leq\frac1{2}(2m\rho\epsilon \mu_2+\frac1{\epsilon}|W^{\Sigma_s}|^2),
        \end{equation}
        \begin{equation}\label{eqWab}
            \suma\sum_{\beta\not=\alpha} W_{\alpha\beta}^{\Sigma_s}\mu_\alpha\mu_\beta \leq(n-1)m\rho(2m\rho\epsilon \mu_2+\frac1{\epsilon}|W^{\Sigma_s}|^2).
        \end{equation}
    \end{lemma}
    \begin{proof}
    Since $\sumb W^{\Sigma_s}_{2\beta}=0$, we have
    \begin{equation*}
        \begin{aligned}
            \sumb W_{2\beta}^{\Sigma_s}\mu_\beta = & \sumb W_{2\beta}^{\Sigma_s}(\mu_\beta-m\rho)\\
            \leq& |W^{\Sigma_s}|\sqrt{\sumb(\mu_\beta-m\rho)^2}\\
            =&|W^{\Sigma_s}|\sqrt{\mu_2(2m\rho-\mu_2)}\\
            \leq&|W^{\Sigma_s}|\sqrt{2m\rho\mu_2}\\
            \leq&\frac1{2}(2m\rho\epsilon \mu_2+\frac1{\epsilon}|W^{\Sigma_s}|^2), 
        \end{aligned}
    \end{equation*}
    which is the first inequality in Lemma~\ref{ineqW}. 
    
    For fixed $\alpha\geq3$, we have $\sum_{\beta\not=\alpha} W^{\Sigma_s}_{\alpha\beta}=0$, so
    \begin{equation*}
        \begin{aligned}
            \sum_{\beta\geq3,\beta\not=\alpha} W_{\alpha\beta}^{\Sigma_s}\mu_\beta = & \sum_{\beta\geq3,\beta\not=\alpha} W_{\alpha\beta}^{\Sigma_s}(\mu_\beta-m\rho) - m\rho W_{\alpha2}^{\Sigma_s}\\
            \leq& |W^{\Sigma_s}|^2\sqrt{\sumb(\mu_\beta-m\rho)^2} - m\rho W_{\alpha2}^{\Sigma_s}\\
            \leq& \frac1{2}(2m\rho\epsilon \mu_2+\frac1{\epsilon}|W^{\Sigma_s}|^2) - m\rho W_{\alpha2}^{\Sigma_s}. 
        \end{aligned}
    \end{equation*}
    Hence
    \begin{equation*}
        \begin{aligned}
            &\suma\sum_{\beta\not=\alpha} W_{\alpha\beta}^{\Sigma_s}\mu_\alpha\mu_\beta \\
            =& 2\mu_2\sumb W_{2\beta}^{\Sigma_s}\mu_\beta + \sum_{\alpha=3}^n\mu_\alpha\sum_{\beta\geq3,\beta\not=\alpha} W_{\alpha\beta}^{\Sigma_s}\mu_\beta\\
            \leq& 2\mu_2\sumb W_{2\beta}^{\Sigma_s}\mu_\beta + \frac1{2}\sum_{\alpha=3}^n \mu_\alpha(2m\rho\epsilon \mu_2+\frac1{\epsilon}|W^{\Sigma_s}|^2) - m\rho\sum_{\alpha=3}^n W_{2\alpha}^{\Sigma_s}\mu_\alpha\\
            =&(2\mu_2-m\rho)\sumb W_{2\beta}^{\Sigma_s}\mu_\beta +  \frac1{2}\sumb \mu_\beta(2m\rho\epsilon \mu_2+\frac1{\epsilon}|W^{\Sigma_s}|^2)\\
            \leq& (\mu_2 +\frac1{2}m\rho + \frac1{2}\sumb\mu_\beta)(2m\rho\epsilon \mu_2+\frac1{\epsilon}|W^{\Sigma_s}|^2)\\
            \leq& \frac{n}{2}(2m\rho\epsilon \mu_2+\frac1{\epsilon}|W^{\Sigma_s}|^2). 
        \end{aligned}
    \end{equation*}
    \end{proof}

    Recall that $|\nabla \varphi |^2  = (\frac{m|\nabla u|}{u})^2 = m^2(\frac{B}{u^2}-K) = m^2(B e^\frac{2\varphi}{m} - K)$ turns to $+\infty$ as $\varphi\rightarrow +\infty$. Define $D(a) = \{\varphi\leq a\}$, then for sufficiently large $a$, $\frac1{|\nabla\varphi|^2}<\delta$ can be arbitrarily small in $M\backslash D(a)$. So \eqref{inNP2_1} together with \eqref{eqWab} implies:

    \begin{proposition} There exists a constant $C = C(n,m,\lambda,\epsilon)>0$, such that
        \begin{equation}\label{inNP2_2}
            |\nabla P|^2 \leq C\mu_2 + C|W^{\Sigma_s}|^2 + \frac{8}{n-3}m^2\rho^2(K_{12}-\rho),
        \end{equation}
        holds in $M\backslash D(a)$ for some sufficiently large $a$. 
    \end{proposition}

\section{Uniform decay on $\mu_1+\mu_2$}\label{sec5}

    Further estimate on the Weyl tensor of level sets in this paper needs 4-dimensional Gauss-Bonnet-Chern formula, so now take $n=5$. 

    \begin{lemma}
        Let $(M^5, g, u, \lambda)$ be a nontrivial simply-connected compact 5-dimensional $m$-quasi-Einstein manifold with boundary and constant $R=\frac{3m+5}{m+1} \lambda$, $m>1$. Then
        \begin{equation}\label{inintW}
            \int_{{\Sigma_s}}|W^{{\Sigma_s}}|^2 d\sigma = 2\int_{{\Sigma_s}}\frac{|\nabla_{\nabla\varphi}P|^2}{|\nabla\varphi|^4}d\sigma\leq 2\int_{{\Sigma_s}}\frac{|\nabla P|^2}{|\nabla\varphi|^2}d\sigma, 
        \end{equation}
        where ${\Sigma_s}$ is level set $\varphi^{-1}(s)$ for some regular value $s$. 
    \end{lemma}
    \begin{proof}
        Recall that for a closed oriented 4-dimensional Riemannian manifold $(\Omega^4,g)$, Gauss-Bonnet-Chern formula implies that
        \begin{equation*}
            \chi(\Omega) =\frac1{4\pi^2}\int_\Omega (\frac1{8}|W_g|^2-\frac1{4}|\Ric_g|^2 +\frac1{12}(R_g)^2) dV. 
        \end{equation*}

        The critical set of $u$ is $N^3$, so as the deformation retraction of $M$, $N$ is also simly-connected. Therefore, $N$ is diffeomorphism to $\mathbb{S}^3$, and the regular level set $\Sigma_s$, as tubular neighborhood of $N$, is diffeomorphism to $\mathbb{S}^1\times \mathbb{S}^3$ (see \cite{costa2024rigiditycompactquasieinsteinmanifolds}), and $\chi({\Sigma_s})=0$. So using \eqref{eqR1a1b}, \eqref{eqsc} and \eqref{eqric} one gets

        \begin{equation*}
        \begin{aligned}
            \int_{\Sigma_s}|W^{\Sigma_s}|^2d\sigma = & \int_{\Sigma_s} \Big(2|\Ric^{\Sigma_s}|^2-\frac2{3}(R^{\Sigma_s})^2\Big)d\sigma\\
            = & \int_{\Sigma_s} \Big(2\suma\Big((1+\frac1{m})\mu_\alpha-\frac{\nabla_{\nabla\varphi}P_{\alpha\alpha}}{|\nabla\varphi|^2}\Big)^2  \\
            &\qquad + 2\sum_{\alpha\not=\beta}\Big(\frac{\nabla_{\nabla\varphi}P_{\alpha\beta}}{|\nabla\varphi|^2}\Big)^2-6(m+1)^2\rho^2 \Big)d\sigma\\
            =& \int_{\Sigma_s} \Big(2(1+\frac1{m})^2\suma \mu_\alpha^2+2\frac{|\nabla_{\nabla\varphi} P|^2}{|\nabla\varphi|^4} - 6(m+1)^2\rho^2 \Big)d\sigma\\
            =& 2\int_{{\Sigma_s}}\frac{|\nabla_{\nabla\varphi}P|^2}{|\nabla\varphi|^4}d\sigma. 
        \end{aligned}
        \end{equation*}
    \end{proof}

    This lemma combined with the estimate of $|\nabla P|^2$ implies $|W^{\Sigma_s}|^2$ turns to 0 near $\partial M$, which gives more information on $\partial M$. 

    \begin{proposition}
        Let $(M^5, g, u, \lambda)$ be a nontrivial simply-connected compact 5-dimensional $m$-quasi-Einstein manifold with boundary and constant $R=\frac{3m+5}{m+1} \lambda$, $m>1$. Then
        \begin{equation}
            \mu_1+\mu_2 = 0
        \end{equation}
        holds on $\partial M$. 
    \end{proposition}
    \begin{proof}
    The equation \eqref{eqK1a} for $K_{12}$ shows
        \begin{equation*}
            \begin{aligned}
                K_{12}-\rho \leq & \frac1{4} + (K_{12}-\rho)^2\\
                = & \frac1{4} + \Big(\frac{\nabla_{\nabla\varphi}P_{22}}{|\nabla\varphi|^2}-\frac{\mu_2(\mu_2-m\rho)}{|\nabla\varphi|^2}-\frac{\mu_2}{m}\Big)\\
                \leq & \frac1{4} + 3\Big(\frac{|\nabla P|^2}{|\nabla\varphi|^2}+\frac{\mu_2^2(\mu_2-m\rho)^2}{|\nabla\varphi|^4}+\frac{\mu_2^2}{m^2}\Big)\\
                \leq & C\Big(\frac{|\nabla P|^2}{|\nabla\varphi|^2}+\frac1{|\nabla\varphi|^4}+1\Big),
            \end{aligned}
        \end{equation*}
        for some $C>0$. Therefore by \eqref{inNP2_2} there exsits $C>0$, such that
        \begin{equation*}
        \begin{aligned}
            |\nabla P|^2 \leq & C(\mu_2+|W^{\Sigma_s}|^2)+C\Big(\frac{|\nabla P|^2}{|\nabla\varphi|^2}+\frac1{|\nabla\varphi|^4}+1\Big)\\
            \leq & \frac{C}{|\nabla\varphi|^2}|\nabla P|^2 + C\big(|W^{\Sigma_s}|^2+1\big)
        \end{aligned}
        \end{equation*}
        holds in $M\backslash D(a)$ for some sufficiently large $a$. So there exists some $C=C(m,n,\lambda,\epsilon)>0$ such that $|\nabla P|^2\leq C(|W^{\Sigma_s}|^2+1)$ holds in $M\backslash D(a)$ for some sufficiently large $a$. 
        
        Insert this into the inequality \eqref{inintW} to get
        \begin{equation}
            \int_{\Sigma_s}|W^{\Sigma_s}|^2\Big(1-\frac{C}{|\nabla\varphi|^2}\Big)d\sigma \leq \int_{\Sigma_s}\frac{C}{|\nabla\varphi|^2}d\sigma. 
        \end{equation}
        Hence $\int_{\Sigma_s}|W^{\Sigma_s}|^2d\sigma \leq \frac{C}{|\nabla\varphi|^2}$ holds for some $C>0$ and level set ${\Sigma_s}$ of sufficiently large $s$. In particular, the Weyl tensor of the boundary $W^{\partial M} = W^{\Sigma_{+\infty}}=0$, that is, $\partial M$ is locally conformally flat. 

        Check expression for Ricci tensor on level set \eqref{eqric} and notice that $\frac{\nabla_{\nabla\varphi}P_{\alpha\beta}}{|\nabla\varphi|^2}\rightarrow0$ uniformly near $\partial M$, one gets $\Ric^{\partial M}$ has eigenvalues $(1+\frac1{m})\mu_\alpha,\ \alpha\in\{2,3,4,5\}$, which means, $\Ric^{\partial M}\geq0$. 

        Thanks to the classificatioin of complete locally conformally flat manifolds with nonnegative Ricci curvature by Zhu \cite{LCF} and by Carron-Herzlich \cite{LCF2006}, the compact closed manifold $\partial M$ must be one of the following:\\
        (i). globally conformally equivalent to $\mathbb{R}^4$ with non-flat metric;\\
        (ii). globally conformally equivalent to a spaceform of positive curvature;\\
        (iii). locally isometric to the cylinder $\bR \times \mathbb{S}^3$. \\
        (iv). locally isometric to a flat manifold;\\
        $\partial M$ is diffeomorphism to $\mathbb{S}^1\times \mathbb{S}^3$ with non-flat metric, so only (iii) is possible, therefore the minimal eigenvalue of $\Ric^{\partial M}$, that is, $(1+\frac1{m})\mu_2$ must be 0. So we have $\mu_1+\mu_2=0$ on $\partial M$. 
    \end{proof}

\section{Proof of Theorem~\ref{THM}}\label{sec6}

    In this section we will give the proof of Theorem~\ref{THM}. Firstly we need modify the operator $L_{m+2}$ to involve all non-decay one-order entries. In $M\backslash N$, take
    \begin{equation}
    \begin{aligned}\label{h}
        h =& -(m+2)\log u + \frac{3m\rho}{2mK}\log(B-Ku^2) - \\
        &16m^2\rho^2\Big(\frac1{2m^3K^2}\log(B-Ku^2)+\frac{B}{2m^3K^2}\frac1{B-Ku^2}\Big), 
    \end{aligned}
    \end{equation}
    where $|\nabla u|^2 = B-Ku^2$. Then $h$ is constant on each level set of $u$. 

    Since 
    \begin{equation*}
    \begin{aligned}
        \nabla\Big(\frac{1}{2mK}\log(B-Ku^2)\Big) = & -\frac{u}{m|\nabla u|^2}\nabla u\\
        = & \frac{1}{|\nabla\varphi|^2}\nabla\varphi,
    \end{aligned}
    \end{equation*}
    and 
    \begin{equation*}
    \begin{aligned}
        &\nabla\Big(\frac1{2m^3K^2}\log(B-Ku^2)+\frac{B}{2m^3K^2}\frac1{B-Ku^2}\Big)\\
        = & \Big(-\frac1{m^3K}\frac{u}{B-Ku^2} +\frac{B}{m^3K}\frac{u}{(B-Ku^2)^2}\Big)\nabla u \\
        = & \Big(\frac{1}{m^3}\frac{u^3}{(B-Ku^2)^2}\Big)\nabla u\\
        = & - \frac1{|\nabla\varphi|^4}\nabla\varphi
    \end{aligned}
    \end{equation*}
    we have
    \begin{equation}
    \begin{aligned}
        \Delta_h=&e^h\divsymb(e^{-h}\nabla) = \Delta - <\nabla h, \nabla>\\
        =& L_{m+2} - \frac{3m\rho}{|\nabla\varphi|^2}<\nabla\varphi,\nabla>-\frac{16m^2\rho^2}{|\nabla\varphi|^4}<\nabla\varphi,\nabla>. 
    \end{aligned}
    \end{equation}

    \begin{proposition}\label{IneqR}
         Let $(M^5, g, u, \lambda)$ be a nontrivial simply-connected compact 5-dimensional $m$-quasi-Einstein manifold with boundary and constant $R=\frac{3m+5}{m+1} \lambda$, $m>1$. Then for any sufficiently large $a$, 
         \begin{equation}\label{ineqPro6.1}
             \int_{M\backslash D(a)} \Delta_h(\mu_1+\mu_2)e^{-h}dV\leq\int_{M\backslash D(a)} -\frac{\rho}{2}(\mu_1+\mu_2)e^{-h}dV, 
         \end{equation}
         where $h$ is defined as $\eqref{h}$. 
    \end{proposition}
    \begin{proof}
        From the decay condition of $\mu_1+\mu_2$, we can assume that $\mu_3-\mu_2\geq\delta>0$ in $M\backslash D(a)$ for some $a$. This follows from $2m\rho\mu_2\geq\mu_2(2m\rho-\mu_2)=\sumb(\mu_\beta - m\rho)^2\geq(\mu_3-m\rho)^2$. Then at any point $x\in M\backslash D(a)$: if $\mu_2=0$, $P_{22}\geq0$ implies $\nabla P_{22}=0$; if $\mu_2\not=0$, locally we can assume $P_{22}=\mu_2$. Hence in $M\backslash D(a)$, $\nabla P_{22} = \nabla \mu_2$. 

        Now considering \eqref{Lmu2} with \eqref{eqW2b} and taking $\epsilon = 1$, we get
        \begin{equation}
        \begin{aligned}\label{Pp61-1}
            L_{m+2}(\mu_1+\mu_2) \leq& -2(m+1)\rho(\mu_1+\mu_2)+\frac{2(m+1)}{m}\mu_2^2\\
            &+\frac{12m^2\rho^2}{|\nabla\varphi|^2}\mu_2+2m\rho\mu_2+|W^{\Sigma_s}|^2\\
            &+\frac{3m\rho}{|\nabla\varphi|^2}\nabla\varphi\nabla\mu_2+\frac{|\nabla\varphi||\nabla P|}{|\nabla\varphi|^2}\mu_2\\
            \leq& -\rho(\mu_1+\mu_2) + |W^{\Sigma_s}|^2 + \frac{3m\rho}{|\nabla\varphi|^2}\nabla\varphi\nabla(\mu_1+\mu_2), 
        \end{aligned}
        \end{equation}
        in the sense of distribution, in $M\backslash D(a)$. 

        We need a more precise estimate on $\int_{\Sigma_s} |W^{\Sigma_s}|^2 d\sigma$. We can use \eqref{inintW} and \eqref{inNP2_2} to derive:
        \begin{equation*}
            \begin{aligned}
                \int_{\Sigma_s} |W^{\Sigma_s}|^2d\sigma \leq& \int_{{\Sigma_s}} 2\frac{|\nabla P|^2}{|\nabla\varphi|^2}|d\sigma\\
                \leq& \int_{{\Sigma_s}}\frac{C\mu_2 + C|W^{\Sigma_s}|^2+8m^2\rho^2(K_{12}-\rho)}{|\nabla\varphi|^2}d\sigma, 
            \end{aligned}
        \end{equation*}
        so for sufficiently large $a$, on level set ${\Sigma_s}\subset M\backslash D(a)$, 
        \begin{equation}\label{Pp61_2}
            \begin{aligned}
                \int_{\Sigma_s} |W^{\Sigma_s}|^2d\sigma \leq& \int_{{\Sigma_s}}\Big( \frac{C\mu_2 }{|\nabla\varphi|^2} + \frac{16m^2\rho^2(K_{12}-\rho)}{|\nabla\varphi|^2}\Big)d\sigma\\
                \leq& \int_{{\Sigma_s}}\Big( \frac{C^\prime\mu_2 }{|\nabla\varphi|^2} + \frac{16m^2\rho^2\nabla\varphi\nabla\mu_2}{|\nabla\varphi|^4}\Big)d\sigma, 
            \end{aligned}
        \end{equation}
        where in the last inequality we use the equation \eqref{eqK1a} for $K_{12}$ and combine all terms with $\mu_2$. 
        
        Integrating \eqref{Pp61-1}, using $\Delta_h$ and \eqref{Pp61_2}, we have 
        \begin{equation*}
            \begin{aligned}
                \int_{\Sigma_s} \Delta_h(\mu_1+\mu_2)d\sigma \leq& \int_{{\Sigma_s}} \Big(-\rho(\mu_1+\mu_2)+\frac{C^\prime\mu_2 }{|\nabla\varphi|^2}\Big)d\sigma\\
                \leq& \int_{{\Sigma_s}} \Big(-\frac1{2}\rho(\mu_1+\mu_2)\Big)d\sigma, 
            \end{aligned}
        \end{equation*}
        holds for all level set ${\Sigma_s} \subset M\backslash D(a)$, for sufficiently large $a$. 

        Noting that $e^{-h}$ is constant on each level set, we prove the proposition~\ref{IneqR}. 
        
    \end{proof}

    We want to show that the LHS of \eqref{ineqPro6.1}. To do so, we will see that the LHS of \eqref{ineqPro6.1} is related to derivative of $\int_{\Sigma_s}(\mu_1+\mu_2)d\sigma$, and use the decay of $\mu_1+\mu_2$. 
    
    The one-parameter family of diffeomorphisms $F(s,x)$ defined as 
    \begin{equation*}
    \begin{cases}
        \frac{\partial F}{\partial s}=\frac{\nabla\varphi}{|\nabla\varphi|^2},\\
        F(x,a) = x ,\ x\in \Sigma_a, 
    \end{cases}
    \end{equation*}
    gives a local coordinate on $M$, $g_{\alpha\beta} = g(\frac{\partial F}{\partial x^\alpha},\frac{\partial F}{\partial x^\beta})$. So
    \begin{equation}
        \begin{aligned}
            \frac{\partial}{\partial s} d\sigma_{{\Sigma_s}} = &\frac{\partial}{\partial s}\sqrt{det(g_{\alpha\beta})} dx^2\wedge dx^3\wedge dx^4\wedge dx^5\\
            = & \frac1{2}2g^{\alpha\beta}g(\nabla_{\frac{\partial F}{\partial s}}\frac{\partial F}{\partial x^\alpha},\frac{\partial F}{\partial x^\beta})d\sigma_{{\Sigma_s}}\\
            = & g^{\alpha\beta}g(\nabla_{\frac{\partial F}{\partial x^\alpha}}\frac{\partial F}{\partial s},\frac{\partial F}{\partial x^\beta})d\sigma_{{\Sigma_s}}\\
            = & \frac1{|\nabla\varphi|^2}g^{\alpha\beta}\varphi_{\alpha\beta}d\sigma_{{\Sigma_s}}\\
            = & \frac1{|\nabla\varphi|^2}g^{\alpha\beta}(\lambda g_{\alpha\beta} - R_{\alpha\beta} + \frac1{m}\varphi_\alpha\varphi_\beta)d\sigma_{{\Sigma_s}}\\
            = & \frac1{|\nabla\varphi|^2}(4\lambda - (R-R_{11}) )d\sigma_{{\Sigma_s}}\\
            = & \frac{m\rho}{|\nabla\varphi|^2}d\sigma_{{\Sigma_s}}. 
        \end{aligned}
    \end{equation}
    Let $\Phi(s) = -\frac{m\rho}{2mK}\log(1-\frac{K}{B}e^{\frac{-2s}{m}})$, then $\Phi>0$ and
    \begin{equation}
        \frac{\partial}{\partial s}\Big(\Phi(s)d\sigma_{{\Sigma_s}}\Big) = \Big(-\frac{m\rho}{2mK}\frac{2Ku^2}{m|\nabla u|^2}+\frac1{|\nabla\varphi|^2}\Big)d\sigma_{{\Sigma_s}} = 0. 
    \end{equation}
    Now we will use $\int_{\Sigma_s} (\mu_1+\mu_2)\Phi(s)d\sigma$ to get the inequality we want: 
    \begin{proposition}\label{ineqL}
        Let $(M^5, g, u, \lambda)$ be a nontrivial simply-connected compact 5-dimensional $m$-quasi-Einstein manifold with boundary and constant $R=\frac{3m+5}{m+1} \lambda$, $m>1$. Then for $a$ in Proposition~\ref{IneqR}, there exists $b>a$, such that
        \begin{equation}
            \int_{M\backslash D(b)} \Delta_h(\mu_1+\mu_2)e^{-h}dV \geq0. 
        \end{equation}
    \end{proposition}
    \begin{proof}
        Notice that 
        \begin{equation*}
            \begin{aligned}
                \int_{\Sigma_s}<\nabla(\mu_1+\mu_2),\frac{\nabla h}{|\nabla h|}>e^{-h} d\sigma\leq&\int_{\Sigma_s}C|\nabla P|^2e^{-h} d\sigma
            \end{aligned}
        \end{equation*}
        turns to 0 as $s\rightarrow+\infty$. So 
        \begin{equation}\label{ineqL_1}
            \int_{M\backslash D(s)} \Delta_h(\mu_1+\mu_2)e^{-h}dV  = -\int_{\Sigma_s}<\nabla(\mu_1+\mu_2),\frac{\nabla h}{|\nabla h|}>e^{-h} d\sigma
        \end{equation}
    
        Define $I(s) = \int_{\Sigma_s} (\mu_1+\mu_2)\Phi(s)d\sigma$, then
        \begin{equation*}
        \begin{aligned}
            I^\prime(s) =& \int_{\Sigma_s} <\nabla(\mu_1+\mu_2),\frac{\partial F}{\partial s}>\Phi(s)d\sigma\\
            =& \frac{\Phi(s)}{|\nabla\varphi|}\int_{\Sigma_s} <\nabla(\mu_1+\mu_2),\frac{\nabla h}{|\nabla h|}>d\sigma, 
        \end{aligned}
        \end{equation*}
        where $s$ is sufficiently large to ensure $\nabla h$ and $\nabla\varphi$ share the same direction. Note that $\Phi>0$ implies $I(s)\geq0$, and $I(s)\rightarrow0$ as $s\rightarrow+\infty$, so there must be $b\geq a$ such that $I^\prime(b)\leq0$. So RHS of the equation \eqref{ineqL_1} is non-negative when $s=b$. 
    \end{proof}

    Proposition~\ref{IneqR} and Proposition~\ref{ineqL} mean that $\mu_1=\mu_2=0$ outside a finite level set. The rest argument is quite similar to the proof of Theorem 4 in \cite{costa2024rigiditycompactquasieinsteinmanifolds}. 

    \begin{proof}[Proof of Theorem\ref{THM}]
        From Proposition~\ref{IneqR} and~\ref{ineqL} we know that $\mu_1 = \mu_2 = 0,\ \mu_3 = \mu_4 = \mu_5 = m\rho$ in $M\backslash M(b)$. Then from inequality \eqref{Pp61_2} and the equation \eqref{eqK1a}, $W^{\Sigma_s} = 0$ and $K_{12} = \rho$ in $M\backslash M(b)$. Therefore it derives from the estimate \eqref{np2} on $|\nabla P|^2=|\nabla \Ric|^2$ that $\nabla\Ric = 0$ in $M\backslash M(b)$. 

        He, Petersen and Weylie \cite{HPW-2012} have proven that $g$ and $u$ are real analytic under harmonic coordinates on $\text{int}M$, which means $\nabla \Ric =0$ holds in $\text{int}M$. Recall that the second Bianchi identity says $\nabla_l R_{ijkl} = \nabla_j R_{ik} - \nabla_i R_{jk}$. So Riemannian curvature is harmonic, and hence by Corollary 1.14 in \cite{HPW-2014}, $(M,g,u)$ is rigid. 

        By Proposition~\ref{Rigid} that $M$ is simply-connected, $(M,g,u)$ can split as
        $$
        \begin{aligned}
        M & =\left(M_1, g_1\right) \times\left(M_2, g_2\right) \\
        u(x_1,x_2) & =u(x_2) = u_2(x_2),
        \end{aligned}
        $$
        where $(M_1,g_1)$ is $\lambda$-Einstein manifold, $M_2$ is $\mathbb{S}^{n_2}_+$, $u = u_2$. Critical set of $u_2$ is only one point, so dim$M_1$ = dim $N$ = $k$ = 3, which means $M_1 = \mathbb{S}^3$. 
        
    \end{proof}

\section*{Acknowledgements}
    The author would like to thank Professor Haizhong Li for guidance and encouragement, and Jingche Chen for valuable comments and for checking the computations.
    
\bibliography{references}
\end{document}

%% file: head.tex
\usepackage{amsmath}
\usepackage{amssymb}
\usepackage{amsthm}
\usepackage[msc-links]{amsrefs}

\usepackage{hyperref}
\usepackage[noabbrev,capitalize]{cleveref}
\usepackage{mathrsfs}
\usepackage{mathtools}
\usepackage{slashed}
\usepackage{tikz-cd}
\usepackage[shortlabels]{enumitem}

\setenumerate{label=(\roman*)}

\DeclareMathOperator{\divsymb}{div}

\DeclareMathOperator{\tr}{tr}

\DeclareMathOperator{\Ric}{Ric}

\newcommand{\bR}{\mathbb{R}}

\def\sideremark#1{\ifvmode\leavevmode\fi\vadjust{\vbox to0pt{\vss
 \hbox to 0pt{\hskip\hsize\hskip1em
 \vbox{\hsize3cm\tiny\raggedright\pretolerance10000
 \noindent #1\hfill}\hss}\vbox to8pt{\vfil}\vss}}}

\newtheorem{theorem}{Theorem}[section]
\newtheorem{proposition}[theorem]{Proposition}
\newtheorem{lemma}[theorem]{Lemma}
\newtheorem{corollary}[theorem]{Corollary}

\theoremstyle{definition}
\newtheorem{definition}[theorem]{Definition}

\newtheorem{example}[theorem]{Example}

\theoremstyle{remark}
\newtheorem{remark}[theorem]{Remark}

\numberwithin{equation}{section}